%% file: ATAC.tex
\RequirePackage{fix-cm}
\documentclass[smallextended]{svjour3}       % onecolumn (second format)
\smartqed  % flush right qed marks, e.g. at end of proof
  \usepackage[nocompress]{cite}
  \usepackage{amsmath,amsfonts}
  
  \usepackage{array}
 \usepackage{comment}
\usepackage{colortbl}
\usepackage{hhline}
\usepackage{cellspace}
\usepackage{multirow,bigdelim}
\usepackage{paralist}
\usepackage{pgfplotstable}
\usepackage{pgfplots}
\usepackage{varwidth}
\usepackage{tabularx}
\usepackage{tabulary}
\usepackage[nomargin,inline,index]{fixme}
\usepackage{tikz}
\usepackage{pgf}
\usepackage{xxcolor}
\usetikzlibrary{arrows,shadows,petri}
\usepackage{optidef}
\usepackage{enumitem}
\usepackage{authblk}
\usepackage{mathtools}
\usepackage{nicefrac}
\usepackage{booktabs}
\usepackage{multirow}
\usepackage{amsmath}

\newtheorem{defi}{Definition}

\newtheorem{thm}[defi]{Theorem}
\newtheorem{lem}[defi]{Lemma}
\newtheorem{cor}[defi]{Corollary}

\begin{document}

\title{Optimal Data Distribution for Big-Data All-to-All Comparison using Finite Projective and Affine Planes 
}

\titlerunning{All-to-All Comparison}        % if too long for running head

\author{Joanne L. Hall* \and Wayne A. Kelly \and Yu-Chu Tian}

\authorrunning{J.L. Hall, W.A. Kelly, Y.-C. Tian} % if too long for running head

\institute{
           J.L. Hall \at
           School of Science \\
           RMIT University\\          
           Tel.: +61-3-9925-2511\\
           \email{joanne.hall@rmit.edu.au}\\
          ORCID: 0000-0003-4484-1920
           \and  
			W.A. Kelly, Y.-C. Tian \at
           School of  Computer Science\\
           Queensland University of Technology\\
           \email{w.kelly@qut.edu.au, y.tian@qut.edu.au}\\
           ORCID: 0000-0002-8554-4589, 0000-0002-8709-5625
}

\date{Received: date / Accepted: date}

\maketitle

\begin{abstract}
An All-to-All Comparison problem is where every element of a data set is compared with every other element. This is analogous to projective planes and affine planes where every pair of points share a common line.   

For large data sets, the comparison computations can be distributed across a cluster of computers.  All-to-All Comparison does not fit the highly successful Map-Reduce pattern, so a new distributed computing framework is required. The principal challenge is to distribute the data in such a way that computations can be scheduled where the data already lies.  

This paper uses projective planes, affine planes and balanced incomplete block designs to design data distributions and schedule computations. The data distributions based on these geometric and combinatorial structures achieve minimal data replication whilst balancing the computational load across the cluster.    %Problems with the All-to-All Comparison pattern appear in applications including computational biology and image processing.    
\end{abstract}

\input{01-Introduction.tex}
\input{02-Relatedwork.tex}
\input{03-Lowerbound.tex}

\input{06-Projective.tex}

\input{07-Affine.tex}
\input{07-BIBD.tex}
\input{09-LoadBalance.tex}
\input{10-Conclusion.tex}

%\input{11-Bibliography.tex}

%\begin{acknowledgements}
%\end{acknowledgements}

\bibliographystyle{spmpsci}      % mathematics and physical sciences
\bibliography{ATACbib.bib}   % name your BibTeX data base

%\begin{thebibliography}{}
% etc
%\end{thebibliography}

\section*{Ethical Considerations}
On behalf of all authors, the corresponding author states that there is no conflict of interest.

\end{document}

%% file: 01-Introduction.tex
\section{Introduction}

{\em All-to-All Comparison} (ATAC) is a computation in which every element of a data set is pairwise compared with every other element \cite{zhang2016data, zhang2017scalable}.   The well known geometric structures of {\em projective planes} and {\em affine planes}\cite{storme2007finite}, where every pair of points share a common line,  and  the combinatorial structures of balanced incomplete block designs (BIBD)\cite{mathon2006}, where every symbol appears in specified number of blocks, are examples of ATAC designs.  We leverage the geometric and combinatorial structures to create ATAC designs suitable for use in scheduling large ATAC computations.

All-to-All Comparison problems are a challenge with large datasets, so we assume that the data set is very large.
Rather than working with data items individually, we group the data into equivalence classes called {\em data groups}.  
A {\em data group} is defined as a set of data items that are all present on the  same set of machines.

An ATAC data distribution may be visualised using a Venn diagram, as in Figure \ref{figc}, or using incidence matrices, as in Definition \ref{def:dist}.

\begin{figure}%[t]
	\centering
	\begin{tikzpicture}
%	\node at (0.8,2.5) {A};
%	\node at (4.2,2.5) {B};
%	\node at (1.6,0.2) {C};
	\draw (2,2) circle [radius=1];
	\draw (3,2) circle [radius=1];
	\draw (2.5,1) circle [radius=1];
	
	\draw [fill=red] (2.5,2.65) circle[radius=0.05];
	\draw [fill=red] (2.3,2.4) circle[radius=0.05];
	\draw [fill=red] (2.5,2.4) circle[radius=0.05];
	\draw [fill=red] (2.7,2.4) circle[radius=0.05];
	\draw [fill=red] (2.2,2.2) circle[radius=0.05];
	\draw [fill=red] (2.4,2.2) circle[radius=0.05];			
	\draw [fill=red] (2.6,2.2) circle[radius=0.05];	
	\draw [fill=red] (2.8,2.2) circle[radius=0.05];	
	
	\draw [fill=red] (1.9,1.65) circle[radius=0.05];
	\draw [fill=red] (1.7,1.4) circle[radius=0.05];
	\draw [fill=red] (1.9,1.4) circle[radius=0.05];
	\draw [fill=red] (2.1,1.4) circle[radius=0.05];
	\draw [fill=red] (1.6,1.2) circle[radius=0.05];
	\draw [fill=red] (1.8,1.2) circle[radius=0.05];			
	\draw [fill=red] (2.0,1.2) circle[radius=0.05];	
	\draw [fill=red] (2.2,1.2) circle[radius=0.05];

	\draw [fill=red] (3.1,1.65) circle[radius=0.05];
	\draw [fill=red] (2.9,1.4) circle[radius=0.05];
	\draw [fill=red] (3.1,1.4) circle[radius=0.05];
	\draw [fill=red] (3.3,1.4) circle[radius=0.05];
	\draw [fill=red] (2.8,1.2) circle[radius=0.05];
	\draw [fill=red] (3.0,1.2) circle[radius=0.05];			
	\draw [fill=red] (3.2,1.2) circle[radius=0.05];	
	\draw [fill=red] (3.4,1.2) circle[radius=0.05];

	\end{tikzpicture}
		\caption{An ATAC data distribution on three machines.  The 24 data items (red dots) have been partitioned into three data groups of 8 data items).   Each data group is placed on two machines (large circles). \label{figc} } 
	%\vspace{5mm}
%	\begin{tabular}{cc|c|c|c|}
%	 & Size ($\vec{s}$) & Machine A & Machine B & Machine C \\ \hline
%		Data Group 1 & $\nicefrac{1}{3}$ & 1 & 1 & 0 \\ \hline
%		Data Group 2 & $\nicefrac{1}{3}$ & 1 & 0 & 1 \\ \hline
%		Data Group 3 & $\nicefrac{1}{3}$ & 0 & 1 & 1 \\	 \hline	
%		\multicolumn{2}{c|}{Total data:} & $\nicefrac{2}{3}$ & $\nicefrac{2}{3}$ & $\nicefrac{2}{3}$  \\ [2mm]
%	\end{tabular}
\end{figure}
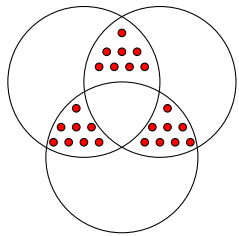	
The objects being compared could be texts\cite{spiccia2016semantic}, genes \cite{guyon2009comparison, zhang2016data,  zhang2017scalable}, or chemical compounds  \cite{alawneh2020scalable, wong2014aligning}.   
When working with  large data sets, the comparison tasks can be distributed across several machines.  
%To ensure efficient computation, the distribution of data and the scheduling of computations are expected to maintain {\em data locality}.
To maximise efficiency, the data required for each computation needs to be on the local machine for which that computation is scheduled, called {\em data locality}. To achieve data locality for an ATAC problem a data distribution is needed such that, for every pair of data items, there exists at least one machine that holds both data items. 
  
Data locality could be achieved by placing every data item on every machine, however, this is not practical with large datasets.  A large amount of network bandwidth and storage space is required to replicate and distribute data.  Unnecessary network traffic, and unnecessary data storage is an unnecessary cost to the organisation doing the computation. % Reducing the amount of data replication reduces the cost of the ATAC computation. 
The cost saving of reducing data replication could be measured in time needed to transfer the data, or the monetary cost of network operations.

A good ATAC data distribution minimizes the amount of data replication and balances computational load so that each machine is doing the same amount of computation. % in a distributed computing environment.
%The physical infrastructure of a compute cluster consisting of computers and network is often shared with other users. Therefore,  minimizing the amount of data distributed reduces network traffic and the storage requirements on each machine.
{\em Load balancing} across all machines means that all machines finish at approximately the same time, thus minimizing the time taken for the overall ATAC computation.

MapReduce frameworks are widely used for distributed processing of large data sets. There are mature implementations of MapReduce frameworks optimised for a variety of use cases \cite{hadoop,spark}. 
However, problems with the ATAC pattern do not fit the MapReduce pattern\cite{zhang2016data}. 
A new compututational framework for distributing data and scheduling computations for  ATAC problems is needed.

\begin{defi}\cite{zhang2016data, zhang2017scalable}
An {\em ATAC} data distribution has every pair of data items together on some machine. %and every machine has the same number of data items.
\end{defi}

%\begin{problem}{\em Distribute data across a network so that each comparison computation has data locality and  each machine has an equal share of the data whilst minimising data replication.}
%\end{problem}

The ATAC data distribution problem was first formally described using graphs and incidence matrices \cite{zhang2016data, zhang2017scalable}. Heuristic approaches \cite{zhang2016data, zhang2017scalable}, tabu search \cite{deng2021solution} and mixed integer linear programming \cite{li2019construction}, have been used to design ATAC data distributions. In this work, we use the tools of discrete geometry and combinatorial designs to design good ATAC data distributions. 

Let $m$ be the number of machines available for computation, and let $x$ be the number of groups that the data has been partitioned into. A data distribution can be expressed as an $m\times x$ incidence matrix $D$, and a column vector $\vec{s}$ of length $x$.   All entries in $\vec{s}$ are non-negative and sum to $1$.
The fraction of all data on each machine may be represented by a column vector $\vec{f}$ of length $m$ and computed by the matrix vector equation 
\begin{equation}
D  \vec{s}=\vec{f}.
\end{equation}

A good ATAC data distribution minimises the amount of data replication and is load balanced so that each machine finishes at the same time.  The quality of a data distribution can be measured by the  data storage limit $L(D,\vec{s})$, which is the maximum fraction of the data stored on any machine.

\begin{defi}\label{def:dist}
Let $(D,\vec{s})$ be a data distribution on $m$ machines, and let $\vec{f}=(f_1,f_2,\dots,f_m)^T=D\vec{s}$, then the {\em data storage limit} $L(D,\vec{s})$ is  
\[L(D,\vec{s})=max_{1\leq i\leq m}(f_i).\]
\end{defi}
The data distribution shown in Figure \ref{figc} can be represented as the matrix vector pair $(D,\vec{s})$ where
\begin{equation}
D\vec{s}=\left(\begin{matrix}
1 & 1 & 0 \\
1 & 0 & 1\\
0 & 1 & 1\\
\end{matrix}\right) 
\left(\begin{matrix} 
\nicefrac{1}{3}\\
\nicefrac{1}{3}\\
\nicefrac{1}{3}\\
\end{matrix}\right)= \left(\begin{matrix} \nicefrac{2}{3}\\
\nicefrac{2}{3}\\
\nicefrac{2}{3}
\end{matrix}\right).
\end{equation}
The data limit of the Figure \ref{figc}  data distribution is $L(D,\vec{s})=\nicefrac{2}{3}$, which we will show is minimal for a cluster of 3 machines (Theorem \ref{thm:bound}).

Our goal is to design  ATAC data distributions which have minimum possible  data storage limit. The load balancing aspect of an ATAC data distribution will be considered in section \ref{LoadBalance}.

\begin{defi}
Let $m$ be the number of machines available for an ATAC data distribution.  The {\em minimal data storage limit}, $L_{min}(m)$, is the minimum data storage limit that can be achieved with a data distribution that uses $m$ machines.
\end{defi}

In the case of a MapReduce data distribution,  it is easy to determine the minimum data storage limit: the data set is partitioned evenly,  for a cluster of  $m$ machines, each machine receives $\nicefrac{1}{m}$ of the data.
For an ATAC data distribution, the minimum data storage limit is more complex.

This paper demonstrates a theoretical lower bound on the ATAC data storage limit.  By using the geometric structures of {\em projective planes}, {\em affine planes}, and {\em balance incomplete block designs} this paper finds ATAC designs with minimum data storage limit for some specific compute cluster sizes.

\input{Graph.tex}

Figure~\ref{Graph1} provides an overview of this paper's results by summarizing the data storage limits achieved by the techniques discussed.
%The remainder of this paper is organized as follows. 

Section \ref{Relatedwork} reviews related work, including earlier heuristic approaches. 
The red squares in Figure~\ref{Graph1} show the heuristic solutions \cite{zhang2017scalable}. 
Section~\ref{Bound} formalizes the ATAC problem in terms of matrix-vector equations and demonstrates a lower bound on the data storage limit, illustrated by black line in Figure~\ref{Graph1}).

Section~\ref{FPP} constructs ATAC data distributions using finite projective planes and affine planes which have optimal data storage limit, illustrated by the black dots in Figure~\ref{Graph1}. Section \ref{sec:BIBD} explores some balanced incomplete block Designs to construct good ATAC data distributions, illustrated by green triangles in Figure \ref{Graph1}. %Section \ref{SmallM} uses computational search methods for ATAC data distributions for small numbers of  machines,
% illustrated by blue crosses in Figure \ref{Graph1}. 
%Section \ref{Extension} explores how optimal data distributions can be used as a basis for efficiently generating quality heuristic solutions for values between the known optimal cases.
%These are illustrated by the green line in Figure \ref{Graph1}.
Load balancing and fault tolerance are discussed in Section~\ref{LoadBalance}. 
Section \ref{sec:conclusion} concludes the paper. 

%% file: Graph.tex
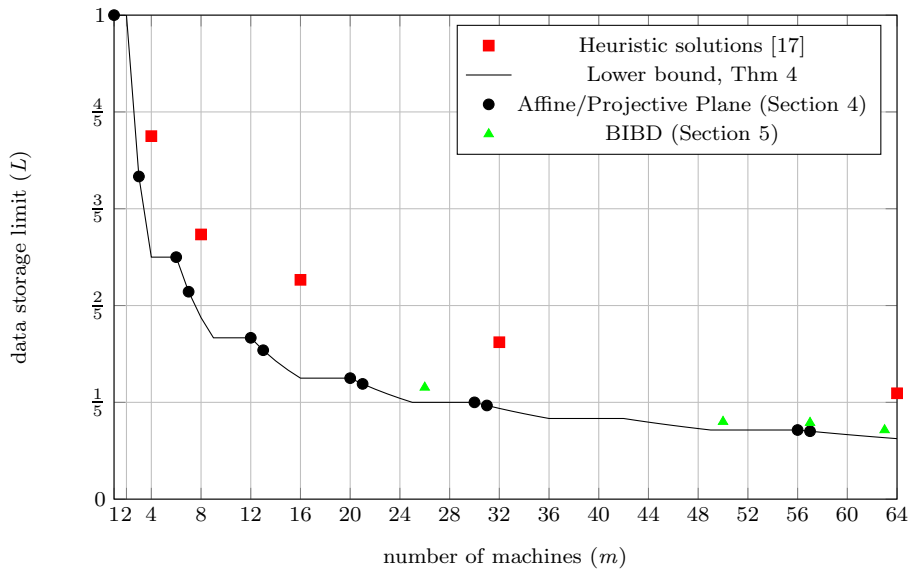
\begin{figure*}[htb!]
\begin{tikzpicture}
\begin{axis}[xlabel=number of machines (\textit{m}), ylabel=data storage limit (\textit{L}), grid=major, height=8cm, width=\textwidth, yticklabel style={/pgf/number format/frac}, xmin = 1, xmax=64, ymin=0, ymax=1, xtick={1,2,4,8,12,16,20,24,28,32,36,40,44,48,52,56,60,64}]

\addplot[only marks,color=red,mark=square*,domain=1:64] coordinates { (4,192/256) (8,140/256) (16,116/256) (32, 83/256) (64, 56/256)};
\addlegendentry{Heuristic solutions \cite{zhang2017scalable}};

\addplot[color=black,domain=1:64] coordinates {(1,1) (2,1) (3,2/3) (4,1/2) (5,1/2) (6,1/2) (7,3/7) (8,3/8) (9,1/3) (10,1/3) (11,1/3) (12,1/3) (13,4/13) (14,2/7) (15,4/15) (16,1/4) (17,1/4) (18,1/4) (19,1/4) (20,1/4) (21,5/21) (22,5/22) (23,5/23) (24,5/24) (25,1/5) (26,1/5) (27,1/5) (28,1/5) (29,1/5) (30,1/5) (31,6/31) (32,3/16) (33,2/11) (34,3/17) (35,6/35) (36,1/6) (37,1/6) (38,1/6) (39,1/6) (40,1/6) (41,1/6) (42,1/6) (43,7/43) (44,7/44) (45,7/45) (46,7/46) (47,7/47) (48,7/48) (49,1/7) (50,1/7) (51,1/7) (52,1/7) (53,1/7) (54,1/7) (55,1/7) (56,1/7) (57,8/57) (58,8/58) (59,8/59) (60,8/60) (61,8/61) (62,8/62) (63,8/63) (64,1/8) };
\addlegendentry{Lower bound, Thm \ref{thm:bound}}

%\addplot[color=blue,mark=x,domain=1:64] coordinates {(1,1) (2,1) (3,2/3) (4,3/5) (5,5/9) (6,1/2) (7,3/7) (8,5/12) (9,2/5) (10,3/8) (11,5/14) (12,1/3) (13,4/13) };
%\addlegendentry{Optimal for small \textit{m} (Section \ref{SmallM})}

\addplot[only marks,color=black,domain=1:64] coordinates {(1,1) (3,2/3) (6,1/2) (7,3/7) (12,1/3) (13,4/13) (20,1/4) (21,5/21) (30,1/5) (31,6/31) (56,1/7) (57,8/57) };
\addlegendentry{Affine/Projective Plane (Section \ref{FPP})}

\addplot[only marks,color=green,mark=triangle*, domain=1:64] coordinates {(26,3/13) (50,4/25) (57,3/19) (63,4/28)  };
\addlegendentry{BIBD (Section \ref{sec:BIBD})}

%\addplot[color=green,domain=1:64] coordinates { (2,1) (3,2/3) (4,0.6) (5,0.555556) (6,1/2) (7,3/7) (8,0.416667) (9,0.4) (10,0.382353) (11,0.363636) (12,1/3) (13,4/13) (14,0.304348)  (15,0.300000)  (16,0.294574)  (17,0.285464)  (18,0.278195)  (19,0.273574) (20,1/4) (21,5/21) (22,0.236842)  (23,0.235294)  (24,0.230769)  (25,0.227642)  (26,0.225248)  (27,0.222222)  (28,0.218525)  (29,0.215848) (30,1/5) (31,6/31) (32,0.192982) (33,0.192308) (34,0.190476) (35,0.188443) (36,0.186678) (37,0.185705) (38,0.184730) (39,0.182688)  (40, 0.180780) (41,0.179802) (42,0.178502) (43,0.176561) (44,0.175439) (45,0.174030) (46,0.172615) (47,0.171492) (48,0.170489) (49,0.169148) (50,0.167932) (51,0.166558) (52,0.165552) (53,0.164608) (54,0.163748) (55,0.162732)  (56,1/7) (57,8/57) (58,0.140187) (59,0.140000) (60,0.139541) (61, 0.139104) %(62, 0.2) (63, 0.2) (64, 0.2)} ;
%\addlegendentry{Extension Heuristic (Section \ref{Extension})}

\end{axis}
\end{tikzpicture}
\label{graph}
\caption{Data per machine vs Number of machines\label{Graph1}.  The red squares are data distributions calculated in existing work~\cite{zhang2017scalable}. The black line is the theoretical lower bound achieved in this paper. The black dots are data distributions constructed from affine and projective planes, see section \ref{FPP}. The green triangles are data distributions constructed from balanced incomplete block designs, see section \ref{sec:BIBD}. %The blue crosses are optimal data distributions found using exhaustive computation, see section \ref{SmallM}.  
Compared with existing work~\cite{zhang2017scalable}, the new data distributions from this paper show significantly better data storage limit.} 
\end{figure*}

%% file: 02-Relatedwork.tex
\section{Background}\label{Relatedwork}
A {\em comparison function} takes two inputs.  In the All-to-All Comparison pattern, the comparison function is applied independently to each {\em pair} of data items; the results are then combined into a similarity matrix.
The  MapReduce framework, as implemented in Hadoop and Spark, has a {\em map function} that takes one input, and a {\em reduce}  function that combines the mapped data into a summarized form.  The key difference between Comparison and the MapReduce framework is that the map function in MapReduce takes only one item of input whereas  the compare function in ATAC takes a pair of input items.
This difference together with the desire to perform tasks where the data already exists makes efficient data distribution for an ATAC computation profoundly more challenging and interesting.

The use of discrete geometry in  distributed computing problems is not new \cite{suen1992efficient, kashefi2017rp2}.  However this is the first applied to ATAC problems.

The ATAC problem has been investigated for a few years, with the problem formalised in 2016 \cite{zhang2016data,zhang2017scalable}.  There have been some recent innovations  to finding good ATAC data distributions \cite{deng2021solution, li2019construction}, however 
there remains challenging issues and no general theory. 
%It is still an open problem whether or not finding an ATAC data distribution with minimum data limit is NP-complete. 
%Similar problems have been shown to  be NP-complete, for example, completion of Balanced Incomplete Block Designs \cite{Kaski2006} and  embedding of Steiner Triple systems \cite{Colbourn1983}.  
%Other similar problems such as isomorphism of line graphs \cite{Huber2011} and construction of Steiner Triple Systems  \cite{gibbons2007} are not NP-complete. 
In Section \ref{FPP}, ATAC data distributions with minimal data limit are constructed for an infinite family of special cases.%Investigating the NP-completeness of the general case remains as future work.

In existing work~\cite{zhang2016data,zhang2017scalable, deng2021solution}, it was assumed that calculating an optimal data distribution for the all-to-all comparison pattern is a hard problem and so heuristic solutions were used.
The red squares in Figure \ref{Graph1} show the heuristic solutions obtained in \cite{zhang2017scalable}. 
By comparison, this paper demonstrates a theoretical lower bound on the data limit, and uses geometric and combinatorial structures to create ATAC data distributions that meet the lower bound.  The ATAC distributions constructed in this paper have  lower data limit than heuristic solutions, however are only constructable for specific compute cluster sizes.

An ATAC design is an \emph{pair covering} \cite{chee2013covering} with the relaxation that block size is not assumed to be constant across the entire design. Making ATAC a new twist on a well studied problem.

The ATAC problem can be formalised into an objective function with multiple constraints which can be used for mixed integer linear programming \cite{li2019construction}.  In Section \ref{Bound} matrix vector equations are used to  formalise the ATAC problem, and a lower bound for the objective function is demonstrated.

In the ATAC pattern under consideration in this paper, the machines in the compute cluster are assumed to be homogeneous, whereas previous work  has considered  heterogeneous systems~\cite{zhang2016data}.
This paper constructs  data distributions with minimal data limit first, and then considers computational load balancing, whereas existing work has considered data distribution and load balancing simultaneously \cite{zhang2016data,zhang2017scalable}.  The advantage of using highly regular geometric  and combinatorial structures, is that load balancing is always possible. 

%% file: 03-Lowerbound.tex
\section{Data Locality and Storage Limit}
\label{Bound}
%\input{VennDiagram.tex}
%%All-to-All Comparison problems are a challenge with large datasets, so we assume that the data set is very large.
%Rather than working with data items individually, we group the data into equivalence classes called {\em data groups}.  
%A {\em data group} is defined as a set of data items that are all present on the  same set of machines.

A data distribution can be expressed as the pair $(D, \vec{s})$.  For  $m$ machines, and $x$ data groups, the matrix $D=(d_{ij})\in\mathbb{R}^{m\times x}$ has $m$ rows (representing machines) and  $x$ columns (representing data groups), with $d_{ij}=1$ when data group $j$ is present on machine $i$, and $d_{ij}=0$ otherwise.  
The vector $\vec{s}=(s_i)\in\mathbb{R}^{x}$ represents the partitioning of the data, where $s_i$ is the fraction of all data in data group $i$, and $\sum_{i=1}^{x}s_i=1$.

The fraction of all data on each machine is represented by a vector $\vec{f}\in \mathbb{R}^m$ computed by the matrix vector equation 
\begin{equation}
D  \vec{s}=\vec{f}.
\end{equation}
The data locality requirement stated for data groups is:
\begin{quote}
\emph{For every pair of data groups there exists at least one machine that holds both groups}.
\end{quote}
The data locality requirement stated in matrix terms is:
\begin{quote}
\emph{The matrix $A=D^TD$ contains all non-zero entries.} 
\end{quote}  

%With $m$ machines, there are $2^m$ possible sets of machines, and hence $2^m$ possible data groups.  In practice, many of the data groups are empty, so only the the all zero columns of $D$ and zero entries of $\vec{s}$ are omitted. 
  
An advantage of using data groups over individual data items is that the optimal data distribution depends only on the size of the compute cluster ($m$) 
and not on the number of items in the data set. 
When defining the size of each data group we use the fraction of all data items that are in that group.
This fraction is a rational number, however, the actual number of data items in each data group must be an integer.
Assuming that the number of data items is sufficiently large, then  rounding the number of items in each data group to the nearest integer makes a negligible difference.

The quality of the data distribution is measured by the {\em data storage limit} $L$, which is determined by the machine that is allocated the most data:
\begin{equation}
L(D,\vec{s}) = \max (D \vec{s})
\end{equation}
It is most useful to describe the data storage limit, $L$, in terms of a fraction of the whole data set.

The challenge is: given $m$, with $x$ unknown, find a zero-one $m \times x$ matrix $D$ and a vector $\vec{s}$ of length $x$ such that:
\begin{itemize}
\item[1)]  $D^TD$ has all nonzero entries,
\item[2)]  the entries of $\vec{s}$ sum to $1$, and
\item[3)] $L(D,\vec{s})$ is minimized.
\end{itemize}

The minimal data storage limit for a cluster of $m$ machines is denoted $L_{min}(m)$.  We establish a lower bound on the data storage limit for a given number of machines. 

\begin{thm}\label{thm:bound}
A lower bound on the data storage limit for an ATAC design using $m$ machines is given by 
\begin{equation}
 \min\left(\frac{\left\lfloor\sqrt{m}\right\rfloor+1}{m},\frac{1}{\lfloor\sqrt{m}\rfloor}\right)\le L_{min}(m)
\label{general}
\end{equation}
\end{thm}

The proof of this theorem  follows directly from Lemmas \ref{lem:bound} and \ref{lem:equiv}.

%Two lemmas are given below before the statement of the main result . 

\begin{lem}\label{lem:bound}
A bound on the data storage limit for and ATAC design using $m$ machines is given by
\begin{equation}\label{eqn:b}
\max_{b \in [1 .. m]}\left( \min\left(\frac{b+1}{m},\frac{1}{b}\right)\right) \le L_{min}(m)
%\min\left(\frac{b+1}{m},\frac{1}{b}\right) \le L(m),\quad \forall b \in [1 .. m]
\end{equation}
\end{lem}
\begin{proof}
For an arbitrary number $b \in [1,..,m]$, consider two alternative cases.
\paragraph{Case 1:} There exists at least one data item that is present on at most $b$ machines. 
That data item needs to be compared with all other data items on one of those $b$ machines, so these $b$ machines collectively hold all data items. 
The average fraction of all data that these $b$ machines hold is at least  $\frac{1}{b}$, so at least one of those machines will hold at least $\frac{1}{b}$ of the data.  
Therefore
\[
\frac{1}{b} \le L_{min}(m).
\]

\paragraph{Case 2:} All data items are present on at least $b+1$ machines. 
For $n$ data items, the total number of data items distributed is at least $(b+1)n$. 
Thus the average data items per machine is $(b+1)n/m$.
Therefore at least one machine  holds at least $(b+1)/m$ of the $n$ data items, therefore
\[
\frac{b+1}{m} \le L_{min}(m).
\]
Combining the two alternative cases:
\[
 \frac{b+1}{m} \le L_{min}(m) \; \lor \; \frac{1}{b} \le L_{min}(m) 
\]
so:
\[
\min\left(\frac{b+1}{m},\frac{1}{b}\right) \le L_{min}(m),\quad \forall b \in [1 .. m]
\]
%this completes the proof. 
as required.  \qed
\end{proof}

%The next Lemma simplifies the left side of equation (\ref{eqn:b}).

\begin{lem}\label{lem:equiv}
Let $m\in\mathbb{N}$ and let 
$f_m(b) = \min\left(\frac{b+1}{m},\frac{1}{b}\right)$then
\[
\max_{b \in [1..m]} f_m(b) = f_m\left(\left\lfloor\sqrt{m}\right\rfloor\right).
\]
\end{lem}
\begin{proof}

Since $\left\lfloor\sqrt{m}\right\rfloor \in [1..m]$, it remains only to show that $f_m$ has its maximum value at $\left\lfloor\sqrt{m}\right\rfloor$. We consider $b$ within four ranges:
\paragraph{Case 1:}
Let $b=\sqrt{m}$ or $b=\sqrt{m}-1$, then $f_m(b)=\frac{1}{\sqrt{m}}$.

\paragraph{Case 2:}
Let $b<\sqrt{m}-1$, then $b=\sqrt{m}-1-\alpha$ for some $\alpha\in \mathbb{R}$ with $\alpha>0$.  
$$\frac{b+1}{m} = \frac{\sqrt{m}-1-\alpha+1}{m} = \frac{1}{\sqrt{m}}-\frac{\alpha}{m} < \frac{1}{\sqrt{m}}$$
and
$$\frac{1}{b} =\frac{1}{\sqrt{m}-1-\alpha} > \frac{1}{\sqrt{m}}$$
$$\frac{b+1}{m} < \frac{1}{b} \textnormal{ , so } f_m(b) = \frac{b+1}{m} < \frac{1}{\sqrt{m}}$$

\paragraph{Case 3:}  
Let  $\sqrt{m}-1 < b < \sqrt{m}$, 
then  $b=\sqrt{m}-\gamma$, with $\gamma\in\mathbb{R}$ with $0 < \gamma < 1$.  
\[\frac{1}{b}=\frac{1}{\sqrt{m}-\gamma}>\frac{1}{\sqrt{m}},\]
and 
\[\frac{b+1}{m}=\frac{\sqrt{m}-\gamma+1}{m}>\frac{1}{\sqrt{m}}.\]
So when $\sqrt{m}-1 < b < \sqrt{m}$ we have $f_m(b) > \frac{1}{\sqrt{m}}$.

\paragraph{Case 4:} 
Let $b>\sqrt{m}$, then $b=\sqrt{m}+\beta$ for some $\beta\in\mathbb{R}$ with $\beta>0$.
$$\frac{b+1}{m} = \frac{\sqrt{m}+\beta+1}{m}  = \frac{1}{\sqrt{m}}+\frac{\beta+1}{m} > \frac{1}{\sqrt{m}}$$
and
$$\frac{1}{b} =\frac{1}{\sqrt{m}+\beta} <\frac{1}{\sqrt{m}}$$
$$\frac{1}{b} < \frac{b+1}{m} \textnormal{ , so } f_m(b) = \frac{1}{b} < \frac{1}{\sqrt{m}}$$

If $\sqrt{m}$ is an integer, then there are no integer values of $b$ covered by case 4, so the next best (case 3) indicates that $f_m$  takes its maximum value at $\sqrt{m} = \left\lfloor\sqrt{m}\right\rfloor$ (and also at $\sqrt{m}-1$). 
If $\sqrt{m}$ is not an integer, then the maximum value of $f_m$ is covered by case 4 and the only integer value of $b$ within that range is $\left\lfloor\sqrt{m}\right\rfloor$.
\qed
\end{proof}

The lower bound on the data storage limit of Theorem \ref{thm:bound} is plotted in black in Figure \ref{Graph1}.
In Section \ref{FPP} data distributions are constructed which   demonstrate that the bound of Theorem \ref{thm:bound} is tight for two families of solutions.

%% file: 06-Projective.tex
\section{ Projective  and Affine Planes}
\label{FPP}

In this section we use the geometric structures of projective and affine planes  to generate  ATAC data distributions with minimal data limit.

\begin{defi}\label{defi:PP}\cite{storme2007finite} A {\em finite  projective plane}  of order $n$ is a set of $m$ points and $m$ lines (where $m=n^2+n+1$) with properties which include:
 \begin{enumerate}
 \item{\label{PP1}Every pair of points has exactly one line passing through both points,}
 \item{Every pair of lines intersects at exactly one of the points,}
 \item{Every line contains $n+1$ points, and}
 \item{Every point is on $n+1$ lines.} 
\end{enumerate}
\end{defi}
An incidence matrix has rows which represent the lines, and columns which represent the points.  The cell in row $i$, column $j$ has value $1$ if line $i$ contains point $j$, all other cells have value $0$.

\begin{thm} \label{thm:PG}
Let $D$ be the incidence matrix of a projective plane of order $n$, 
let $m=n^2+n+1$, and $\vec{s}=(\nicefrac{1}{m}, \nicefrac{1}{m}, \cdots, \nicefrac{1}{m})^T\in\mathbb{R}^m$. Then, $(D,\vec{s})$ is an ATAC data distribution on $m$ machines with minimal data limit.
\end{thm}

\begin{proof}
View the points of the projective plane as data groups and the lines as machines.  Property \ref{PP1} of projective planes (Definition \ref{defi:PP}) ensures that the incidence matrix of a projective plane provides a data distribution that satisfies  data locality, and hence is an ATAC design. 
It remains to show that the data distribution has minimal data storage limit.

The vector $\vec{s}$ represents $m=n^2+n+1$ data groups of size $\nicefrac{1}{m}$; and each machine has $n+1$ data groups; 
so the total data on each machine is $L = \frac{n+1}{n^2+n+1}$. 
Since $m=n^2+n+1$ the value of  $\left\lfloor\sqrt{m}\right\rfloor = n$. 
Substituting this into Theorem \ref{thm:bound} gives: 
$$\min\left(\frac{n+1}{n^2+n+1},\frac{1}{n}\right)\le L_{min}(m).$$
As $\frac{n+1}{n^2+n+1} \le \frac{1}{n}$, and therefore $L = \frac{n+1}{n^2+n+1} \le L_{min}(m)$,
 the data distribution constructed from the projective plane has optimal data storage limit. 
\qed 
\end{proof}

Unfortunately, finite projective planes of many orders do not exist. For example: order $n=6$ corresponding to $m=43$ and order $n=10$, corresponding to $m=111$.
The non-existence of a finite projective plane of order $n=10$ was famously demonstrated using thousands of hours of supercomputer time \cite{lam1991search}.  
The non-existence of projective planes of order $6, 14$, and $21$ is due to the much celebrated Bruck-Ryser-Chowla Theorem~\cite{BruckRyser1977}. 
The possible existence of finite projective planes of certain orders, for example $n=12  (m=157)$, $n=18 (m=343)$ and $n=20 (m=421)$  remain unknown.

For $n=p^r$, where $p$ is a prime and $r$ is a positive integer, then a finite projective plane of order $p^r$ exists, denoted $PG(2, p^r)$.   Polynomial arithmetic over the finite field $GF(p^r)$ can be used to construct $PG(2,p^r)$ \cite{storme2007finite}.
Therefore, there exists an infinite family of ATAC data distributions with minimal data limit which can be efficiently constructed.

%% file: 07-Affine.tex
%\section{Finite Affine Planes}
%\label{APP}

Finite affine planes  are closely related to finite projective planes. 
\begin{defi} \cite{storme2007finite}
A {\em finite affine plane} of order $n$ is a set of $n^2$ points and $n^2+n$ lines with properties which include:
 \begin{enumerate}
 \item{\label{AP1}Every pair of points has exactly one line passing through both points};
 \item{Every line contains $n$ points}; and
 \item{Every point is on $n+1$ lines}.
\end{enumerate}
\end{defi}

\begin{thm} \label{thm:AG}
Let $D$ be the incidence matrix of an affine  plane of order $n$, let $m=n^2+n$, and $\vec{s}=(\nicefrac{1}{n^2}, \nicefrac{1}{n^2}, \cdots, \nicefrac{1}{n^2})^T\in\mathbb{R}^{n^2}$. Then, $(D,\vec{s})$ is an ATAC data distribution  on $m$ machines with minimal data limit.
\end{thm}

\begin{proof}
View the points of the affine plane as data groups and the lines as machines: property \ref{AP1} of affine planes ensures that the incidence matrix of an affine plane provides a data distribution that satisfies data locality and hence is an ATAC design. 
It remains to show that the data distribution has optimal data storage limit.

The vector $\vec{s}$ represents that each data group is of size $\nicefrac{1}{n^2}$; and each machine has $n$ data groups; 
so the total data on each machine is $L = \nicefrac{1}{n}$. Since
$m=n^2+n$ the value $\left\lfloor\sqrt{m}\right\rfloor = n$. 
Substituting this into Theorem \ref{thm:bound} gives: 
$$\min\left(\frac{n+1}{n^2+n},\frac{1}{n}\right)\le L_{min}(m).$$
As $\frac{n+1}{n^2+n} = \frac{1}{n}$, and  therefore $L = \frac{1}{n} \le L_{min}(m)$,
 the data distribution constructed by an affine plane has optimal data storage limit.
\qed 
\end{proof}

If a finite projective plane of order $n$ exists, then a finite affine plane of order $n$ also exists and can be constructed by removing one line and the points on that line from the projective plane.    Polynomial arithmetic over the finite field $GF(p^r)$ can be used to construct affine planes of prime power order, denoted  $AG(2,p^r)$ \cite{storme2007finite}.

So optimal data distributions are easily constructed for $m=n^2+n+1$ (finite projective plane) or $m=n^2+n$ (finite affine  plane), for any prime power $n=p^r$. 
Table~\ref{tab:pg} shows optimal data distributions constructed by affine and projective planes for up to $1000$ machines.

\begin{table}[htb!]
\begin{center}
\caption{\label{tab:pg} Optimal ATAC data distributions using projective and affine planes for systems up to $993$ machines. }
\begin{tabular}{cccccc}
\toprule
%\multirow{2}{*}{$n$} & \multirow{2}{*}{$p^r$} & \multirow{2}{*}{Existence?} & $m=$ & Projective & Affine\\
\multirow{2}{*}{$n$} & \multirow{2}{*}{$p^r$} & {Existence?} & $m=$ & Projective & Affine\\
& & \cite{storme2007finite,BruckRyser1977} & $n^2+n+1$ & $L_{min}(m)$ & $L_{min}(m-1)$\\ \midrule %\bottomrule
1  & $1^1$   & yes &   3 & $\nicefrac{2}{3} \approx 0.67$    & 1  \\ %\midrule
2  & $2^1$   & yes &   7 & $\nicefrac{3}{7} \approx 0.43$    & $\nicefrac{1}{2} = 0.5$  \\ %\midrule
3  & $3^1$   & yes &  13 & $\nicefrac{4}{13} \approx 0.31$   & $\nicefrac{1}{3} \approx 0.33$  \\ %\midrule
4  & $2^2$   & yes &  21 & $\nicefrac{5}{21} \approx 0.24$  & $\nicefrac{1}{4} = 0.25$  \\ %\midrule
5  & $5^1$   & yes &  31 & $\nicefrac{6}{31} \approx 0.19 $  & $\nicefrac{1}{5} = 0.2$  \\ \midrule
6  &         & no  &  43 &  $\times$      &  $\times$    \\ %\midrule
7  & $7^1$   & yes &  57 & $\nicefrac{7}{57} \approx 0.12$   & $\nicefrac{1}{7} \approx 0.14$  \\ %\midrule
8  & $2^3$   & yes &  73 & $\nicefrac{8}{73} \approx 0.11$   & $\nicefrac{1}{8} \approx 0.13$  \\ %\midrule
9  & $3^2$   & yes &  91 & $\nicefrac{9}{91} \approx 0.099$    & $\nicefrac{1}{9}\approx 0.11$  \\ %\midrule
10 &         & no  & 111 &  $\times$      &  $\times$   \\ \midrule
11 & $11^1$  & yes & 133 & $\nicefrac{11}{133} \approx 0.083$ & $ \nicefrac{1}{11} \approx 0.091$ \\ %\midrule
12 &         & unknown & 157 &  ?      &  ? \\ %\midrule
13 & $13^1$  & yes & 183 & $\nicefrac{13}{183} \approx 0.071$ & $\nicefrac{1}{13} \approx 0.077$ \\ %\midrule
14 &         & no  & 211 &   $\times$     &   $\times$   \\ %\midrule
15 &         & unknown & 241 &   ?     &   ?   \\ \midrule
16 & $2^4$   & yes & 273 & $\nicefrac{16}{273} \approx 0.059$ & $\nicefrac{1}{16} \approx 0.063$ \\ %\midrule
17 & $17^1$  & yes & 307 & $\nicefrac{17}{307} \approx 0.055$ & $\nicefrac{1}{17} \approx 0.059$   \\ %\midrule
18 &         & unknown & 343 &   ?     &  ?   \\ %\midrule
19 & $19^1$  & yes & 381 & $\nicefrac{19}{381} \approx 0.050$ & $\nicefrac{1}{19} \approx  0.053$\\ %\midrule
20 &         & unknown & 421 &   ?     &    ?  \\ \midrule
21 &         & no  & 463 &   $\times$     &  $\times$    \\ %\midrule
22 &         & no  & 507 &   $\times$     &  $\times$    \\ %\midrule
23 & $23^1$  & yes & 553 & $\nicefrac{23}{553} \approx  0.042$ & $\nicefrac{1}{23} \approx 0.043$ \\ %\midrule
24 &         & unknown & 601 &   ?     & ? \\ %\midrule
25 & $5^2$   & yes & 651 & $\nicefrac{25}{651} \approx 0.038$  & $\nicefrac{1}{25} \approx 0.040$ \\ \midrule
26 &         & unknown & 703 &   ?     &   ?   \\ %\midrule
27 & $3^3$   & yes & 757 & $\nicefrac{27}{757}  \approx 0.036$  & $\nicefrac{1}{27} \approx 0.037$ \\ %\midrule
28 &         & unknown & 813 &   ?     &  ?   \\ %\midrule
29 & $29^1$  & yes & 871 & $\nicefrac{29}{871}  \approx 0.033$ & $\nicefrac{1}{29} \approx  0.034$ \\ %\midrule
30 &         & no  & 931 &   $\times$     &  $\times$    \\ \midrule
31 & $31^1$  & yes & 993 & $\nicefrac{31}{993} \approx 0.031$ & $\nicefrac{1}{31} \approx  0.032$ \\ \bottomrule
\end{tabular}
\end{center}
\end{table}

%% file: 07-BIBD.tex
\section{Balanced Incomplete Block Designs \label{sec:BIBD}}

Balanced incomplete block designs are a slightly more general structure than the  projective and affine planes explored in Section \ref{FPP}.

\begin{defi}\cite{mathon2006} A {\em balanced incomplete block design} (BIBD) is a set of $v$ symbols and $m$ blocks, such that each block contains $k$ symbols, each symbol appears exactly $r$ times, and each pair of symbols appears together in exactly $\lambda$ blocks.  A BIBD may be denoted as a $t-(v,k,\lambda)$ design.
\end{defi}

Let the symbols of the BIBD represent data groups, and the blocks represent machines: In an ATAC design based on a BIBD, there are $v$ data groups, $m$ machines, each machine holds $k$ data groups, each data group is distributed to $r$ machines, and each pair of data groups appear on exactly $\lambda$ machines.

Note that it is common to use $'b'$ to denote the number of blocks, however we have chosen $'m'$ for consistency with the number of machines.

The parameters of a BIBD design are highly constrained  \cite{BruckRyser1977, mathon2006} with 
\begin{equation}
    vr=mk \label{eqn:vr}
\end{equation} and 
\begin{equation}
r(k-1)=\lambda(v-1)\label{eqn:rk}
\end{equation}
BIBD can be specified using $v$, $k$, and $\lambda$, and are often written as $2-(v,k,\lambda)$ designs, the $2$ referring to the pairs of elements that are balanced.  Since we aim to minimise data replication, a good ATAC design has $\lambda=1$. Note that  Projective and Affine planes as explored in Section \ref{FPP} are specific examples of $2\!-\!(v,k,1)$ Designs; Projective planes are $2\!-\!(n^2+n+1, n, 1)$ designs, and Affine Planes are $2\!-\!(n^2, n, 1)$ designs.

Let $D$ be the incidence matrix of a $2-(v,k,1)$ design, and  $\vec{s}=(\nicefrac{1}{v}, \nicefrac{1}{v}, \cdots, \nicefrac{1}{v})^T\in\mathbb{R}^v$, then $(D,\vec{s})$ is an ATAC design and 
\[D\vec{s}=(\nicefrac{k}{v}, \nicefrac{k}{v}, \cdots, \nicefrac{k}{v})^T=\vec{f}\in\mathbb{R}^b.\]

The data limit of an ATAC design based on a $2-(v,k,1)$ is $L=\nicefrac{k}{v}$.
From equations (\ref{eqn:vr},\ref{eqn:rk}),  the number of machines used in an ATAC design based on a $2-(v,k,1)$ design is $m=\frac{v(v-1)}{k(k-1)}$.

We have already seen  in section \ref{FPP} that projective and affine planes construct ATAC designs with minimal data limit.  We now show that there are other parameters for which BIBD exist and construct good ATAC designs.

\begin{lem} \label{lem:1onm}
Let $D$ be a $2-(v,k,1)$ design with $m=\frac{v(v-1)}{k(k-1)}$ blocks, then when $v=k, k^2,$ or $k(2k-1)$:
\begin{equation}
\frac{k}{v}=\frac{1}{\left\lfloor\sqrt{m}\right\rfloor}. \label{eqn:casestar}    
\end{equation}

\end{lem}
\begin{proof}
Note that from equation (\ref{eqn:vr}) we know that   $m=\nicefrac{vr}{k}$. Rearranging equation (\ref{eqn:casestar}) becomes
%\begin{equation}
%x=\left\lfloor\sqrt{xr}\right\rfloor 
%\end{equation}
\begin{equation}
\frac{v}{k}=\left\lfloor\sqrt{\frac{vr}{k}}\right\rfloor 
\end{equation}
Hence, $\nicefrac{v}{k}\in \mathbb{Z}$ and  $(\nicefrac{v}{k})^2\leq (\nicefrac{v}{k})r <(\nicefrac{v}{k}+1)^2$. Thus $r=\nicefrac{v}{k},\nicefrac{v}{k}+1$, or $\nicefrac{v}{k}+2$.

Using equation (\ref{eqn:rk}) we find that $v=k$, $v=k^2$, or $v=k(2k-1)$. \qed
\end{proof}
If $v=k$, then the ATAC design is the trivial case of all data on one machine.
When $v=k^2$, the ATAC design is an affine plane as described in Theorem \ref{thm:AG},  which has minimal data limit.  

When $v=k(2k-1)$, then $m=4k^2-1$, so $\nicefrac{\left\lfloor\sqrt{m}\right\rfloor+1}{m}<\nicefrac{1}{\left\lfloor\sqrt{m}\right\rfloor}$, the $2-(k(2k-1),k,1)$ designs have a low data limit, but not quite at the minimal limit given by Theorem \ref{thm:bound}.

\begin{lem}\label{lem:monm}
Let $D$ be a $2-(v,k,1)$ design with $m=\frac{v(v-1)}{k(k-1)}$ blocks, and each symbol repeated $r=\frac{v-1}{k-1}$ times, then when $r\leq \sqrt{m}+1$:
\begin{equation}\label{eqn:casebox}
\frac{k}{v}=\frac{\left\lfloor\sqrt{m}\right\rfloor+1}{m}.
\end{equation}
when $r\leq \sqrt{m}+1$.
\end{lem}
\begin{proof}
%Substituting in the expansion for $m$ and rearranging equation (\ref{eqn:casebox})
%\begin{equation}\label{eqn:k}
%\frac{(v-1)}{(k-1)}  =  \left\lfloor\sqrt{\frac{v(v-1)}{k(k-1)}}\right\rfloor +1 
%\end{equation}
Using equation (\ref{eqn:rk}) we can substitute $\nicefrac{r}{m}=\nicefrac{k}{v}$ into equation \ref{eqn:casebox}:
\begin{equation}\label{eqn:k}
r  =  \left\lfloor\sqrt{m}\right\rfloor+1.
\end{equation}
Using equation \ref{eqn:rk} we can substitute $r=\frac{v-1}{k-1}$ into equation (\ref{eqn:k}):
\begin{equation}\label{eqn:k2}
r  =  \left\lfloor\sqrt{\frac{v}{k}r}\right\rfloor +1 
\end{equation}
which puts bounds on the value of $\nicefrac{v}{k}$ in terms of $r$.  \begin{equation}
    (r-1)^2\leq \nicefrac{v}{k}r<r^2
\end{equation}
Expanding, and rewriting using equation (\ref{eqn:rk}) and (\ref{eqn:vr}) we get two inequalities.
\begin{align}\label{inequality<}
 \nicefrac{v}{k} & <\frac{v-1}{k-1}  \\
 \label{inequalityleq}
(r-1)^2 & \leq m
\end{align}

From equation (\ref{eqn:rk}), $\nicefrac{v}{k}<\frac{v-1}{k-1}$ whenever $v>k$.  All BIBD have $v>k$, so all $BIBD$ satisfy   inequality (\ref{inequality<}).

%For  inequality (\ref{inequalityleq}) rearrange to 
%\begin{equation}
%r^2-2r+1\leq m
%\end{equation}
%which is true whenever $r\geq 1+\sqrt{m}$.

Hence $2\!-\!(k,v,1)$ designs satisfy equation (\ref{eqn:casebox}) when $(r-1)^2\leq m$. \qed
\end{proof}

% Projective planes of order $n$ are $2-(n^2+n+1, n+1,1 )$ designs, and affine planes of order $n$ are $2(n^2, n, 1)$ designs.  
Whilst finite fields can be leveraged to construct affine and projective planes of various sizes, there is no known way to construct general $2\!-\!(v,k,1)$ designs so easily.  There are published tables of known $2\!-\!(v,k,1)$ designs \cite{mathon2006} from which an ATAC design could be constructed.  

\paragraph{:} 
\begin{itemize}
    \item A $2\!-\!(13,3,1)$ design has $m=26$ blocks.   The data limit is $L=\nicefrac{\left\lfloor\sqrt{m}\right\rfloor+1}{m}=\frac{k}{v}=\nicefrac{3}{13}=0.23$, which is a little higher than the lower bound of $L_{min}=0.2$ of Theorem \ref{thm:bound}. 

\item A $2\!-\!(19,3,1)$ design has $m=57$ blocks.  Note that $r=9$, and $(r-1)^2=64>57=m$, so this design does not meet the criterion of either Lemma \ref{lem:1onm} or \ref{lem:monm}. The data limit $L=\nicefrac{3}{19}=0.16$ which is higher than the bound of $L_{min}=\nicefrac{1}{7}=0.14$. Note that $PG(2,7)$ also has $m=57$, and has minimal data limit.

\item A $2\!-\!(25,4,1)$ design has $m=50$ blocks. The data limit $L=\nicefrac{8}{50}=0.16$ which is slightly higher than the lower bound of $_{min}L=\nicefrac{1}{7}=0.14$. 

\item  A  $2\!-\!(28,4,1)$ design has $m=63$ blocks.  The  data limit $L=\nicefrac{4}{26}= 0.142$ is slightly higher than the lower bound of $L=\nicefrac{8}{63}=0.126$. 
\end{itemize}

The four examples given here are shown as green triangles in Figure \ref{graph}. Balanced incomplete block designs can be used to construct good ATAC designs. 
ATAC designs based on $2\!-\!(v,k,1)$ designs have low data limit, which in the case of affine an projective planes, is minimal.

%% file: 09-LoadBalance.tex
\section{Load Balancing and Fault Tolerance \label{LoadBalance}}

%The data distributions generated ensure:
%\begin{itemize}
%\item{data locality, meaning that all pairs of data objects can be compared on at least one machine using only local data, and}
%\item{the amount of data replication has been minimized, meaning that every machine has an equally small amount of data that it must store}
%\end{itemize}
We now turn our attention to the problem of scheduling comparison tasks with the goal of {\em load balancing} the compute cluster.  
The computation tasks are scheduled so that each machine does an approximately equal amount of computational work.
%WAK: can't say without loss of generality!
We assume that all machines are equally capable and all comparison tasks take approximately the same amount of time.

If a given pair of data items exists on only a single machine, the comparison task must be scheduled on that machine.
If the pair of data items is available on multiple machines, then scheduling comparison tasks involves some choices. A data distribution is {\em load balanced} if every machine has the same amount of comparison computations.

Formally it is only possible to load balance if the number of computations is divisible by the number of machines.  If no machine is scheduled to undertake any more that 1 more computations than any other machine in the cluster, then the computation is \emph{effectively load balanced}.  In practice, with large data sets, the goal is for effective load balancing.

Heterogeneous compute clusters and file sizes are left as future work.

For data distributions based on affine planes, projective planes and BIBD, load balancing is possible. 
\begin{thm}\label{thm:dist}
Let $(D,s)$  be a data distribution derived from a $2\!-\!(v,k,1)$ design, then an ATAC computation can be effectively load balanced.
\end{thm}
\begin{proof}
There are $m$ machines and $m$ data groups. 
Every data group $g_i$ is compared with every other data group $g_j$, $i, j\in\{1,2,\dots m\} = 1$.
As each of the data groups is of equal size, comparing any distinct pair of data groups takes the same amount of time as comparing any other distinct pair of data groups.
If $g_i \neq g_j$, then there exists precisely one machine that holds both data groups and that machine is allocated those comparison tasks. 
Each machine is therefore  allocated an equal number  of distinct pairs of data groups to compare.

Every item of data group $g_i$ must also be compared with every other item of the same data group $g_i$, $i=1,\cdots,m$.  
%The only non-trivial task allocation problem is deciding on which machine to compare all of the data items within a group with all other data items within that same group.
There  are $r=\nicefrac{v-1}{k-1}$ machines holding group $g_i$ and therefore capable of comparing data within that data group. 
Distribute the intra-group comparison tasks equally amongst those $r$ machines.  As each machine hosts an equal number ($k$) of data groups, load balance is maintained across all machines. \qed
\end{proof}

Projective Planes and Affine Planes are specific cases of $2-(v,k,1)$ designs, so from Theorems \ref{thm:PG}, \ref{thm:AG} and \ref{thm:dist} we have the following corollary:

\begin{cor}
Let $(D,s)$  be a data distribution derived from a projective plane or an affine plane, then an ATAC computation can be load balanced, and the amount of data replication is minimal.
\end{cor}

Given a more general data distribution, the comparison task allocation problem is more complex. There may be distinct pairs of data groups that exist on more than one machine, and data groups are not necessarily of uniform size.

%Given a data distribution $(D,\vec{s})$, it is possible to derive the best possible computational distribution using a linear programming formulation:
%For each pair of groups $i$ and $j$, let real variable $c_{i,j,m}$ represent the fraction of total comparisons that will occur between group $i$ and group $j$ on machine $m$.
%We then minimize the computational load $C$ subject to the following constraints:
%\begin{itemize}
%\item The total work done on each machine is
%$$\forall m, \sum_{i,j} c_{i,j,m} \le C.$$
%If $D_{i,m} = 0$ or $D_{j,m} = 0$ then  group $i$ cannot be compared with group $j$ on machine $m$, so $c_{i,j,m} = 0$. 
%Meaning that in practice the variable $c_{i,j,m}$ can be eliminated from the problem.

%\item The fraction of total comparisons that occur between group $i$ and group $j$  depends on the size of those groups so:
%$$\forall i,j, \sum_m c_{i,j,m} = s_i s_j$$
%\end{itemize}
%It is worth mentioning that once the data distribution $(D,\vec{s})$ is predetermined, $D_{i,m}$ and $s_i$ are constants, not variables.

%All of the variables are real-valued so we can use a polynomial-time linear programming solver to efficiently schedule computational tasks to machines for large clusters (of up to 1000 machines).

An important aspect of big data frameworks is {\em fault tolerance}.
The computation should be able to be completed even if one of the machines in the cluster fails during the computation.
The MapReduce framework achieves fault tolerance by replicating each data item to multiple machines.

In the all-to-all comparison framework, each data item is replicated to several machines. Therefore, if one machines fails, there are other machines from which the data items could be retrieved.
There is no guarantee that every pair of data items  still exist together on at least one machine after failure.  
However the cost of occasionally retrieving data from a remote node is reasonable, as it will allow the computation to complete.
Hence any ATAC design is fault tolerant.

%% file: 10-Conclusion.tex
\section{Conclusion and Future Work \label{sec:conclusion}}

Data locality criteria have been presented in this paper for scalable computation of All-to-All comparison problems of big data in a distributed computing framework. 
The  goal is to minimize the amount of data replication across the system, minimize the amount of storage space needed on each machine, and balance the computational load across the machines. A theoretical lower bound is established on the amount of data stored on each machine. The geometric structures of Projective and Affine planes are used to construct ATAC data distributions that achieve the lower bound for data storage and balance the computational load across the system. Balanced Incomplete Block Designs construct good ATAC designs, though not necessarily optimal.

Affine planes, Projective planes and Balanced Incomplete Block Designs are only applicable in select cases, so further work is required to construct good data distributions for general sizes of compute cluster.
The general problem of constructing good data distributions is computational difficult, but it remains unknown whether it is NP complete.  Future work includes developing techniques for constructing data distributions which may include good heuristic solutions, and considering heterogeneous systems.

%Future work includes determining if the optimal data distribution problem in general is NP-complete and if not, trying to generate an efficient algorithm that works for any number of machines.